\documentclass[11pt,reqno]{amsart}

\usepackage{amsmath,amssymb,mathrsfs}
\usepackage{graphicx,cite,times}

\newtheorem{theo}{Theorem}[section]

\newtheorem{lemm}[theo]{Lemma}

\newtheorem{rema}[theo]{Remark}

\numberwithin{equation}{section}
\def\beq{\begin{equation}}
\def\eeq{\end{equation}}
\def\bea{\begin{eqnarray}}
\def\eea{\end{eqnarray}}
\def\ba{\begin{align}}
\def\ea{\end{align}}

\def\O{\Omega}

\def\na{\nabla}
\def\p{\partial}

\def\ol{\overline}

\begin{document}

\title[increasing stability]{Increasing stability for the inverse source
scattering problem with multi-frequencies}

\author{Peijun Li}
\address{Department of Mathematics, Purdue University, West Lafayette, Indiana
47907, USA.}
\email{lipeijun@math.purdue.edu}

\author{Ganghua Yuan}
\address{KLAS, School of Mathematics and Statistics, Northeast Normal University,
Changchun, Jilin, 130024, China}
\email{yuangh925@nenu.edu.cn}
\thanks{ The research of PL was supported in part by the NSF grant DMS-1151308.
The research of GY was supported in part by NSFC grants 10801030, 11271065,
11571064, the Ying Dong Fok Education Foundation under grant 141001, and the
Fundamental Research Funds for the Central Universities under grant
2412015BJ011.}


\begin{abstract}
Consider the scattering of the two- or three-dimensional Helmholtz equation
where the source of the electric current density is assumed to be
compactly supported in a ball. This paper concerns the stability analysis of the
inverse source scattering problem which is to reconstruct the source function.
Our results show that increasing stability can be obtained for the inverse
problem by using only the Dirichlet boundary data with multi-frequencies.
\end{abstract}

\maketitle

\section{Introduction and problem statement}

In this paper, we consider the following Helmholtz equation:
\begin{equation}\label{sol}
\Delta u(x)+ \kappa^2 u(x)=f(x), \quad x\in\mathbb{R}^{d},
\end{equation}
where $d=2$ or $3$, the wavenumber $\kappa>0$ is a constant, $u$ is
the radiated wave field, and $f$ is the source of the electric current density
which is assumed to have a compact support. Denote
by $B_{\rho}=\{x\in\mathbb{R}^d: |x|<\rho\}$ the ball with radius
$\rho>0$ and center at the original. Let $R>0$ be a constant which is large
enough such that $B_R$ contains the support of $f$. Let $\p B_R$ be the
boundary of $B_R$. The following Sommerfeld radiation condition is required to
ensure the uniqueness of the wave field $u$:
\begin{equation}\label{rc}
\lim_{r\to\infty}r^{\frac{d-1}{2}}(\p_r u-{\rm
i}\kappa u)=0,\quad r=|x|,
\end{equation}
uniformly in all directions $\hat{x}=x/|x|$.

For a given function $u$ on $\p B_R$ in two dimensions, it has the
Fourier series expansion
\[
 u(R, \theta)=\sum_{n\in\mathbb{Z}}\hat{u}_n(R)e^{{\rm i}n\theta},\quad 
\hat{u}_n(R)=\frac{1}{2\pi}\int_0^{2\pi}u(R, \theta)e^{-{\rm
i}n\theta}{\rm d}\theta.
\]
We may introduce the Dirichlet-to-Neumann (DtN) operator $\mathscr{B}:
H^{1/2}(\p B_R)\to H^{-1/2}(\p B_R)$ given by 
\[
 (\mathscr{B}u)(R, \theta)=\kappa\sum_{n\in\mathbb{Z}}
\frac{H_n^{(1)'}(\kappa R)}{H_n^{(1)}(\kappa R)}\hat{u}_n(R) e^{{\rm
i}n\theta}.  
\]
For a given function $u$ on $\p B_R$ in three dimensions, it has the Fourier
series expansion:
\[
 u(R, \theta, \varphi)=\sum_{n=0}^\infty\sum_{m=-n}^n \hat{u}_n^m(R)
Y_n^m(\theta, \varphi),\quad \hat{u}_n^m(R)=\int_{\p B_R} u(R,
\theta, \varphi)\bar{Y}_n^m(\theta, \varphi){\rm d}\gamma.
\]
We may similarly introduce the DtN operator ${\mathscr B}: H^{1/2}(\p B_R)\to
H^{-1/2}(\p B_R)$ as follows:
\[
 (\mathscr{B}u)(R, \theta, \varphi)=\kappa\sum_{n=0}^\infty\sum_{m=-n}^n
\frac{h_n^{(1)'}(\kappa R)}{h_n^{(1)}(\kappa
R)}\hat{u}_n^m(R)Y_n^m(\theta, \varphi). 
\]
Here $H_n^{(1)}$ is the Hankel function of the first kind with order zero,
$h_n^{(1)}$ is the spherical Hankel function of the first kind with order zero,
$Y_n^m$ is the spherical harmonics of order $n$, and the bar denotes the
complex conjuate. Using the DtN operator, we can reformuate the Sommerfeld
radiation condition into a transparent boundary condition 
\[
\p_{\nu}u={\mathscr B}u\quad\text{on} ~ \p B_R,
\]
where $\nu$ is the unit outer normal on $\p B_R$. Hence one can also obtain the
Neumann data on $\p B_R$ once the Dirichlet date is available on $\p B_R$. Now
we are in the position to discuss our inverse source problem:

{\bf IP.} {\em Let $f$ be a complex function with a compact support contained in
$B_R$. The inverse problem is to determine $f$ by using the boundary observation
data $u(x,\kappa)|_{\p B_R}$ with an interval of frequencies $\kappa\in
(0,K)$ where $K>1$ is a positive constant.}

The inverse source problem has significant applications in medical and
biomedical imaging \cite{I-89}, and various tomography problems \cite{Ar,
StUh}. In this paper, we study the stability of the above inverse problem.
As is known, the inverse source problem does not have a unique solution at a
single frequency \cite{DS-IEEE82, HKP-IP05}. Our goal is to establish increasing
stability of the inverse problems with multi-frequencies. We refer to
\cite{BLT-JDE10, CIL} for increasing stability analysis of the inverse source
scattering problem. In \cite{CIL}, the authors discussed increasing stability of
the inverse source problem for the three-dimensional Helmholtz equation in a
general domain $\O$ by using the Huygens principle. The observation data are
both $u(x,\kappa)|_{\p\O}, 0<\kappa<K$ and $\na u(x,\kappa)|_{\p\O},
0<\kappa<K$. In \cite{BLT-JDE10}, the authors studied the stability of the two-
and three-dimensional Helmholtz equations via Green's functions. But the
stabilities in \cite{BLT-JDE10} are different from the stability in this paper
where only the Dirichlet data is required. Related results can be found in
\cite{I-CM07, I-D11} on increasing stability of determining potentials and in
the continuation for the Helmholtz equation. We refer to \cite{El, BaLiTr}
for a uniqueness result and numerical study for the inverse source scattering
problem. A survey can be found in \cite{BLLT-IP15} for some general inverse
scattering problems with multi-frequencies. 

\section{Main result}

Let $0<r<R$, define a complex-valued functional space:
\[
 \mathcal{C}_M =\{f\in H^{n+1}(B_R): \|f\|_{H^{n+1}(B_R)}\leq M, ~{\rm
supp}f\subset B_r\subset B_R, ~f: B_R\to\mathbb{C}\},
\]
where  $M>1$ and  $0<r<R$ are constants. For any $v\in H^{1/2}(\p B_R)$, we set
\[
\|v(x,\kappa)\|_{\p B_R}=\int_{\p B_R}\left( |\mathscr{B} v(x,
\kappa)|^2 +\kappa^2 |v(x, \kappa)|^2 \right){\rm d}\gamma.
\]

Now we show the main stability result of the inverse problem.

\begin{theo}\label{mr1}
Let $f_j\in \mathcal{C}_M, j=1, 2$, and let $u_j$ be the solution of the
scattering problem \eqref{sol}--\eqref{rc} corresponding to $f_j$. Then there
exists a positive constant $C$ independent of $n, K, M, \kappa$ such that
\begin{align}
\label{cfe} \| f_1-f_2\|^2_{L^2(B_R)}\leq C
\left(\epsilon^2+\frac{M^2}{\left(\frac{K^{\frac{2}
{3}}|\ln\epsilon|^{\frac{1}{4}}}{(R+1)(6n-6d+3)^3}\right)^{2n-2d+1}} \right),
\end{align}
where $K>1$, $n\ge d$ and
\begin{align}
\label{e1} \epsilon&=\left(\int_0^K \kappa^{d-1} \|(u_1-u_2)(x,\kappa)\|_{\p
B_R}{\rm d}\kappa\right)^{\frac{1}{2}}.
\end{align}
\end{theo}

\begin{rema}
There are two parts in the stability estimates \eqref{cfe}: the
first part is the data discrepancy and the second part comes from the high
frequency tail of the function. It is clear to see that the stability
increases as $K$ increases, i.e., the problem is more stable as more
frequencies data are used. We can also see that when
$n<\left[\frac{K^{\frac{2}{9}}|\ln\epsilon|^{\frac{1}{12}}}{(R+1)^{\frac{1}{3}}
}+d-\frac{1}{2}\right]$, the stability increases as $n$ increases, i.e., the
problem is more stable as the functions have suitably higher regularity.
\end{rema}

Next we prove Theorem \ref{mr1} in the following section.

\section{Proof of Theorem \ref{mr1}}

First we present several useful lemmas.

\begin{lemm}
 Let $f_j\in L^2(B_R)$ and ${\rm supp}f_j\subset B_R$, $j =1, 2$. Then
 \begin{align*}
 \|f_1-f_2\|^2_{L^2(B_R)}
&\le C\int_0^{\infty}\kappa^{d-1}\int_{\p B_R}\left|\p_{\nu}u(x,\kappa)+\kappa
u(x,\kappa)\right|^2{\rm d}\gamma {\rm d}\kappa.
 \end{align*}
\end{lemm}

\begin{proof}
 Let $\xi\in\mathbb{R}$ with $|\xi|=\kappa$. Multiplying $e^{-{\rm i}\xi
x}$ on both sides of \eqref{sol} and integrating over $B_R$, we obtain
\[
 \int_{B_R}e^{-{\rm i}\xi x}f(x){\rm d}x=\int_{\p B_R}e^{-{\rm i}\xi
x}(\p_{\nu}u(x,\kappa)+{\rm i}\xi\nu u(x,\kappa)){\rm d}\gamma, \quad
|\xi|=\kappa\in (0,\infty).
\]
Since $\mbox{supp}f\subset B_R$, we have
\[
 \int_{\mathbb{R}^d}e^{-{\rm i}\xi x}f(x){\rm d}x=\int_{\p B_R}e^{-{\rm i}\xi
x}(\p_{\nu}u(x,\kappa)+{\rm i}\xi\nu u(x,\kappa)){\rm d}\gamma, \quad
|\xi|=\kappa\in (0,\infty),
\]
which gives
\[
 \left|\int_{\mathbb{R}^d}e^{-{\rm i}\xi x}f(x){\rm
d}x\right|^2\le\left|\int_{\p B_R}(\p_{\nu}u(x,\kappa)+\kappa
u(x,\kappa)){\rm d}\gamma\right|^2, \quad |\xi|=\kappa\in (0,\infty).
\]
Hence,
\begin{align*}
 \left(\int_{\mathbb{R}^d}\left|\int_{\mathbb{R}^d}e^{-{\rm i}\xi
x}f(x){\rm d}x\right|^2 {\rm d}\xi\right)^{\frac{1}{2}}&\cr
 \le&\left(\int_{\mathbb{R}^d}\left|\int_{\p B_R}(\p_{\nu}u(x,\kappa)+\kappa
u(x,\kappa)){\rm d}\gamma\right|^2 {\rm d}\xi\right)^{\frac{1}{2}}, \quad
|\xi|=\kappa\in (0,\infty).
\end{align*}
When $d=2$, we obtain by using the polar coordinates that
\begin{align*}
 &\left(\int_{\mathbb{R}^2}\left|\int_{\mathbb{R}^2}e^{-{\rm i}\xi
x}f(x){\rm d}x\right|^2{\rm d}\xi\right)^{\frac{1}{2}}\cr
 &\le\left(\int_0^{2\pi}{\rm d}\theta\int_0^{\infty}\kappa\left|\int_{\p
B_R}(\p_{\nu}u(x,\kappa)+\kappa
u(x,\kappa)){\rm d}\gamma\right|^2{\rm d}\kappa\right)^{\frac{1}{2 }}\cr
 &\le \left(2\pi\int_0^{\infty}\kappa\left|\int_{\p
B_R}(\p_{\nu}u(x,\kappa)+\kappa
u(x,\kappa)){\rm d}\gamma\right|^2{\rm d}\kappa\right)^{\frac{1}{2}}\cr
 &\le\left(2\pi^2R^2\int_0^{\infty}\kappa\int_{\p
B_R}\left|\p_{\nu}u(x,\kappa)+\kappa u(x,\kappa)\right|^2{\rm d}\gamma
{\rm d}\kappa\right)^{\frac{1}{2}},
\end{align*}
It follows from the Plancherel theorem that
\begin{align*}
 \|f_1-f_2\|^2_{L^2(B_R)}&=\|f_1-f_2\|^2_{L^2(\mathbb{R}^2)}\\
& =\frac{1}{(2\pi)^2}\int_{\mathbb{R}^2}|\hat{f}
_1(\xi)-\hat{f}_2(\xi)|^2{\rm d}\xi\\
&\le C\int_0^{\infty}\kappa\int_{\p B_R}|\p_{\nu}u(x,\kappa)+\kappa
u(x,\kappa)|^2{\rm d}\gamma {\rm d}\xi.
\end{align*}

When $d=3$,  we obtain by using the polar coordinates that
\begin{align*}
 &\left(\int_{\mathbb{R}^3}\left|\int_{\mathbb{R}^3}e^{-{\rm i}\xi
x}f(x){\rm d}x\right|^2{\rm d}\xi\right)^{\frac{1}{2}}\cr
 &\le\left|\int_0^{2\pi}{\rm d}\theta\int_0^{\pi}{\rm sin}\varphi
{\rm d}\varphi\int_0^{\infty}\kappa^2\left|\int_{\p
B_R}(\p_{\nu}u(x,\kappa)+\kappa
u(x,\kappa)){\rm d}\gamma\right|^2{\rm d}\kappa\right|^{\frac{1}{2}}\cr
 &\le \left(2\pi^2\int_0^{\infty}\kappa^2\left|\int_{\p
B_R}(\p_{\nu}u(x,\kappa)+\kappa
u(x,\kappa)){\rm d}\gamma\right|^2{\rm d}\kappa\right)^{\frac{1}{2}}\cr
 &\le\left(\frac{8}{3}\pi^3R^3\int_0^{\infty}\kappa^2\int_{\p
B_R}\left|\p_{\nu}u(x,\kappa)+\kappa u(x,\kappa)\right|^2{\rm d}\gamma
{\rm d}\kappa\right)^{\frac{1}{2}}.
\end{align*}
It follows from the Plancherel theorem that
\begin{align*}
 \|f_1-f_2\|^2_{L^2(B_R)}&=\|f_1-f_2\|^2_{L^2(\mathbb{R}^3)}\\
& =\frac{1}{(2\pi)^3}\int_{\mathbb{R}^3}|\hat{f}
_1(\xi)-\hat{f}_2(\xi)|^2{\rm d}\xi\\
&\le C\int_0^{\infty}\kappa^2\int_{\p B_R}\left|\p_{\nu}u(x,\kappa)+\kappa
u(x,\kappa)\right|^2{\rm d}\gamma {\rm d}\xi,
\end{align*}
which completes the proof.
\end{proof}

For $d=2$, let
\begin{align}\label{3.1}
I_1(s)=&\int_0^s \kappa^{3}\int_{\p B_R}\left(\int_{B_R}-\frac{{\rm
i}}{4}H_0^{(1)}(\kappa|x-y|)(f_1(y)-f_2(y)){\rm d}y\right)\cr
&\left(\int_{B_R}\frac{{\rm
i}}{4}\bar{H}_0^{(1)}(\kappa|x-y|)(\bar{f}_1(y)-\bar{f}_2(y)){\rm
d}y\right){\rm d}\gamma(x){\rm d}\kappa, \\
\label{3.2}I_2(s)=&\int_0^s \kappa\int_{\p
B_R}\left(-\int_{B_R}\frac{{\rm
i}}{4}\p_{\nu}H_0^{(1)}(\kappa|x-y|)(f_1(y)-f_2(y)){\rm d}y\right)\cr
&\left(\int_{B_R}\frac{{\rm
i}}{4}\p_{\nu}\bar{H}_0^{(1)}(\kappa|x-y|)(\bar{f}
_1(y)-\bar{f}_2(y)){\rm d}y\right){\rm d}\gamma(x){\rm d}\kappa.
\end{align}
For $d=3$, let
\begin{align}\label{3.3}
I_1(s)=&\int_0^s \kappa^{4}\int_{\p B_R}\left(\int_{B_R}\frac{e^{{\rm
i}\kappa|x-y|}}{4\pi |x-y|}(f_1(y)-f_2(y)){\rm d}y\right)\cr
&\left(\int_{B_R}\frac{e^{-{\rm i}\kappa|x-y|}}{4\pi
|x-y|}(\bar{f}_1(y)-\bar{f}_2(y)){\rm d}y\right){\rm d}\gamma(x){\rm d}\kappa,\\
\label{3.4}I_2(s)=&\int_0^s
\kappa^{3}\int_{\p B_R}\left(\int_{B_R}\p_{\nu}\frac{e^{{\rm
i}\kappa|x-y|}}{4\pi |x-y|}(f_1(y)-f_2(y)){\rm d}y\right)\cr
&\left(\int_{B_R}\p_{\nu}\frac{e^{-{\rm i}\kappa|x-y|}}{4\pi
|x-y|}(\bar{f}_1(y)-\bar{f}_2(y)){\rm d}y\right){\rm d}\gamma(x){\rm d}\kappa.
\end{align}
Denote
$$ S=\{z=x+{\rm i}y\in\mathbb{C}: -\frac{\pi}{4}<{\rm arg}
z<\frac{\pi}{4}\}.
$$
The integrands in \eqref{3.1}--\eqref{3.4} are analytic functions of
$\kappa$ in $S$. The integrals with respect to $\kappa$ can be taken over any path joining
points $0$ and $s$ in $S$. Thus $I_1(s)$ and $I_2(s)$ are
 analytic functions of $s=s_1+{\rm i}s_2\in S, s_1, s_2\in\mathbb{R}$.

\begin{lemm}
 Let $f_j\in L^2(B_R), {\rm supp}f_j\subset B_R, j =1, 2$. We have  for any
$s=s_1+{\rm i}s_2\in S$ that
\begin{enumerate}

\item for $d=2$,
 \begin{align}
 \label{3.5} |I_1(s)|&\leq 16\pi^3R^3|s|^{5}e^{4R|s_2|}\|f_1(x)-f_2(x)\|_{L^2(B_R)}^2,\\
 \label{3.6} |I_2(s)|&\leq 16\pi^3R^3|s|^{3}e^{4R|s_2|}\|f_1(x)-f_2(x)\|_{H^1(B_R)}^2,
 \end{align}

\item  for $d=3$,
\begin{align}
\label{3.7}  |I_1(s)|&\leq 16\pi^3(|s|^{3}R^3+|s|^{4}R^4)e^{4R|s_2|}\|f_1(x)-f_2(x)\|_{L^2(B_R)}^2,\\
 \label{3.8} |I_2(s)|&\leq 16\pi^3(|s|^{2}R^3+|s|^{3}R^4)e^{4R|s_2|}\|f_1(x)-f_2(x)\|_{H^1(B_R)}^2,
 \end{align}

\end{enumerate}
\end{lemm}

\begin{proof} We first prove (\ref{3.7}). Let $\kappa=st, t\in(0, 1)$. A simple
calculation yields
 \begin{align*}
  I_1(s)=&\int_0^1 s^{5}t^{4}\int_{\p B_R}\left(\int_{B_R}\frac{e^{{\rm
i}st|x-y|}}{4\pi |x-y|}(f_1(y)-f_2(y)){\rm d}y\right)\cr
  &\left(\int_{B_R}\frac{e^{-{\rm i}st|x-y|}}{4\pi
|x-y|}(\bar{f}_1(y)-\bar{f}_2(y)){\rm d}y\right){\rm d}\gamma(x){\rm d}t.
\end{align*}
Noting that $|e^{{\rm i}st |x-y|}|\leq e^{2R|s_2|}$ for all $x\in \p B_R, y\in
B_R$, we have
 \begin{align*}
|I_1(s)|&=\int_0^1|s|^{5}t^{4}\int_{\p
B_R}\left|\int_{B_R}\frac{e^{2|s_2|R}}{|x-y|
}|f_1(y)-f_2(y)|{\rm d}y\right|^2{\rm d}\gamma(x){\rm d}t\cr
&\le\int_0^1|s|^{5}t^{4}\int_{\p B_R}\left|\int_{B_R}
|f_1(y)-f_2(y)|^2 {\rm d}y\right|\int_{B_R}\frac{e^{4R|s_2|}}{|x-y|^2}{\rm d}y
{\rm d}\gamma(x){\rm d}t,\cr
\end{align*}
where we have used the Schwarz inequality for the integral with respect to $y$
in the last inequality. Using the polar coordinates $\rho=|x-y|$ with respect to
$y$ yields
\begin{align*}
|I_1(s)|\le\int_0^1|s|^{5}\left(\int_{B_R}|f_1(y)-f_2(y)|^2{\rm d}y\right)\int_{
\p B_R} \left(2\pi^2\int_0^{2R}e^{4|s_2|R}{\rm d}\rho
\right){\rm d}\gamma(x){\rm d}t,\cr
\end{align*}
which implies (\ref{3.7}).

Next we prove (\ref{3.8}). Let $\kappa=st, t\in(0, 1)$. A simple calculation
yields
 \begin{align*}
  I_2(s)=&\int_0^1 s^{3}t^{2}\int_{\p B_R}\left(\int_{B_R}\p_{\nu}\frac{e^{{\rm
i}st|x-y|}}{4\pi |x-y|}(f_1(y)-f_2(y)){\rm d}y\right)\cr
&\left(\int_{B_R}\p_{\nu}\frac{e^{-{\rm i}st|x-y|}}{4\pi
|x-y|}(\bar{f}_1(y)-\bar{f}_2(y)){\rm d}y\right){\rm d}\gamma(x){\rm d}t,
\end{align*}
which gives
 \begin{align*}
|I_2(s)|=\int_0^1|s|^{3}t^{2}\int_{\p
B_R}\left|\int_{B_R}\na_x\left(\frac{e^{{\rm
i}st|x-y|}}{|x-y|}\right)\cdot\nu(f_1(y)-f_2(y)){\rm
d}y\right|^2{\rm d}\gamma(x){\rm d}t.
\end{align*}
Noting $\na_x\left(\frac{e^{{\rm
i}st|x-y|}}{|x-y|}\right)=-\na_y\left(\frac{e^{{\rm i}st|x-y|}}{|x-y|}\right)$
and ${\rm supp}f_j\subset B_R, j=1,2$, we have
\begin{align*}
  |I_2(s)|=\int_0^1|s|^{3}t^{2}\int_{\p B_R}\left|\int_{B_R}\frac{e^{{\rm
i}st|x-y|}}{|x-y|}\na_y\left(|f_1(y)-f_2(y)\right)\cdot\nu
{\rm d}y\right|^2{\rm d}\gamma(x){\rm d}t.
\end{align*}
Following a similar argument for proving (\ref{3.7}), we can prove (\ref{3.8}).

Now we show the proofs of (\ref{3.5}) and (\ref{3.6}). First we
prove (\ref{3.5}). By (\ref{3.1}) we have
\begin{align*}
  I_1(s)=\int_0^1 s^{4}t^{3}\int_{\p B_R}\left|\int_{B_R}\frac{{\rm
i}}{4}H_0^{(1)}(st|x-y|)(f_1(y)-f_2(y)){\rm d}y\right|^2{\rm d}\gamma(x){\rm
d}t.
\end{align*}
The Hankel function can also be expressed by the following integral when ${\rm
Re}z>0$ (see e.g.,\cite{Wa}, Chapter VI):
\begin{align*}
H_0^{(1)}(z)=\frac{1}{{\rm i}\pi}\int_{1+\infty{\rm i}}^{1}e^{{\rm
i}z\tau}(\tau^2-1)^{-1/2}{\rm d}\tau.
\end{align*}
Consequently,
\begin{align*}
|H_0^{(1)}(z)|&= \left| \frac{1}{\pi}\int^0_{+\infty}e^{{\rm
i}({\rm Re }z+{\rm i}{\rm Im }z)(1+t{\rm i})}((1+t{\rm i})^2-1)^{-1/2}{\rm
d}t\right|\cr
&\le \left|\frac{1}{\pi}e^{{\rm i}{\rm Re }z -{\rm Im
}z}\int^0_{+\infty}e^{-t{\rm Re z}-{\rm i}t{\rm Im }z}(2\tau{\rm
i}-\tau^2)^{-1/2}{\rm d}t\right|\cr
&\le \frac{1}{\pi}e^{|{\rm Im }z|}\int_0^{+\infty}\frac{e^{-t{\rm Re
}z}}{\left|\tau^{1/2}(2{\rm i}-\tau)^{1/2}\right|}{\rm d}t\cr
&\le \frac{1}{\pi}e^{|{\rm Im }z|}\int_0^{+\infty}\frac{e^{-t{\rm Re
}z}}{\tau^{1/2}(\tau^2+4)^{1/4}}{\rm d}t\cr
&\le \frac{1}{\pi}e^{|{\rm Im }z|}\int_0^{+\infty}\frac{e^{-t{\rm Re
}z}}{\tau^{1/2}2^{1/2}}{\rm d}t\cr
&=\frac{1}{\pi}e^{|{\rm Im }z|}\left(\int_0^{1}\frac{e^{-t{\rm Re
}z}}{\tau^{1/2}2^{1/2}}{\rm d}t+\int_1^{+\infty}\frac{e^{-t{\rm Re
}z}}{\tau^{1/2}2^{1/2}}{\rm d}t\right)\cr
&\le\frac{1}{\pi}e^{|{\rm Im }z|}\left(\int_0^{1}\frac{1}{\tau^{1/2}}{\rm
d}t+\int_1^{+\infty}e^{-t{\rm Re }z}{\rm d}t\right)\cr
&\le\frac{1}{\pi}e^{|{\rm Im }z|}\left(2+\frac{1}{{\rm Re }z}\right).
\end{align*}
Similarly, we can obtain
\begin{align*}
|\ol{H}_0^{(1)}(z)|\le \frac{1}{\pi}e^{|{\rm Im }z|}\left(2+\frac{1}{{\rm Re }z}\right).
\end{align*}
Hence we have
\begin{align*}
  |I_1(s)|\le\int_0^1|s|^{4}t^{3}\int_{\p
B_R}\left|\int_{B_R}|f_1(y)-f_2(y)|^2{\rm d}y\right|\int_{B_R}
e^{4R|s_2|}\left(2+\frac{1}{|x-y|s_1t}\right){\rm d}y{\rm d}\gamma(x){\rm d}t.
\end{align*}
Using the polar coordinates $\rho=|x-y|$ with respect to $y$ yields
\begin{align*}
|I_1(s)|\le\int_0^1|s|^{4}t^{3}\left|\int_{B_R}|f_1(y)-f_2(y)|^2{
\rm d}y\right|\int_ { \p B_R}
\left(2\pi^2\int_{0}^{2R}e^{4R|s_2|}\left(2\rho+\frac{1}{s_1t}\right){\rm
d}\rho\right){\rm d}\gamma(x){\rm d}t.
\end{align*}
which completes the proof of (\ref{3.5}).

Noting that $\p_{\nu}H_0^{(1)}(\kappa|x-y|)=\na_x
H_0^{(1)}(\kappa|x-y|)\cdot\nu$ and $\na_x H_0^{(1)}(\kappa|x-y|)=-\na_y
H_0^{(1)}(\kappa|x-y|)$, we can prove (\ref{3.6}) in a similar way.
\end{proof}

\begin{lemm}
Let $f_j\in H^n(B_R), n\ge d, {\rm supp}f_j\subset B_r\subset B_R, j =1, 2$.
Then there exists a constant $C$ independent of $n$ such that for any $s\ge 1$
\begin{align}\label{3.9}
 \int_s^{+\infty} \int_{\p B_R}\kappa^{d-1} \bigl(|\p_{\nu}u(x,\kappa)|^2
+\kappa^2|u(x, \kappa)|^2 \bigr){\rm
d}\gamma{\rm d}\kappa&\leq C s^{-(2n-2d+1)}\| f_1-f_2\|^2_{H^{n+1}(B_R)}.
\end{align}
\end{lemm}

\begin{proof}
It is easy to see that
\begin{align*}
&\int_s^{+\infty}\int_{\p B_R} \kappa^{d-1}\bigl(|\p_{\nu}u(x,\kappa)|^2
+\kappa^2|u(x, \kappa)|^2 \bigr){\rm d}\gamma{\rm d}\kappa\cr
&= \int_s^{+\infty} \int_{\p B_R} \kappa^{d+1}|u(x, \kappa)|^2 {\rm
d}\gamma{\rm d}\kappa  +  \int_s^{+\infty} \int_{\p B_R} \kappa^{d-1}
|\p_{\nu}u(x,\kappa)|^2{\rm d}\gamma{\rm d}\kappa\cr
&\triangleq L_1+L_2.
\end{align*}
Next, we will estimate $L_1$ and $L_2$. When $d=3$, we have
\begin{align*}
L_1&=\int_s^{+\infty} \int_{\p B_R} \kappa^{4}|u(x, \kappa)|^2 {\rm d}\gamma{\rm
d}\kappa\cr
&=\int_s^{+\infty} \int_{\p B_R}
\kappa^{4}\left|\int_{\mathbb{R}^3}\frac{e^{{\rm
i}\kappa|x-y|}}{4\pi|x-y|}(f_1-f_2)(y){\rm d}y\right|^2 {\rm d}\gamma{\rm
d}\kappa.
\end{align*}
Using the polar coordinates $\rho=|y-x|$ originated at $x$ with respect to $y$,
we have
\begin{align*}
L_1=\int_s^{+\infty} \int_{\p B_R} \kappa^{4}\left|\int_0^{2\pi}{\rm
d}\theta\int_0^{\pi}\sin\varphi{\rm d}\varphi\int_0^{+\infty}\frac{e^{{\rm
i}\kappa\rho}}{4\pi}(f_1-f_2)\rho {\rm d}\rho\right|^2 {\rm d}\gamma{\rm
d}\kappa.
\end{align*}
Using integration by parts and noting ${\rm supp}f_j\subset B_r\subset B_R$, we obtain
\begin{align*}
L_1=\int_s^{+\infty} \int_{\p B_R} \kappa^{4}\left|\int_0^{2\pi}{\rm
d}\theta\int_0^{\pi}\sin\varphi{\rm d}\varphi\int_{R-r}^{2R}\frac{e^{{\rm
i}\kappa\rho}}{4\pi({\rm i}\kappa)^n}\frac{\p^n[(f_1-f_2)\rho]}{\p\rho^n} {\rm
d}\rho\right|^2 {\rm d}\gamma{\rm d}\kappa.
\end{align*}
Consequently,
\begin{align*}
L_1\le &\int_s^{+\infty} \int_{\p B_R} \kappa^{4}\left|\int_0^{2\pi}{\rm
d}\theta\int_0^{\pi}\sin\varphi{\rm d}\varphi\int_{R-r}^{2R}\frac{1}{4\pi
\kappa^n}\right.\cr
&\left(\left|\sum\limits_{|\alpha|=n}\p_y^{\alpha}
(f_1-f_2)\right|\rho\right.\left.\left.+n\left|\sum\limits_{|\alpha|=n-1}\p_y^{
\alpha}(f_1-f_2)\right|\right){\rm d}\rho\right|^2 {\rm d}\gamma{\rm d}\kappa\cr
=&\int_s^{+\infty} \int_{\p B_R} \kappa^{4}\left|\int_0^{2\pi}{\rm
d}\theta\int_0^{\pi}\sin\varphi{\rm d}\varphi\int_{R-r}^{2R}\frac{1}{4\pi
\kappa^n}\right.\cr
&\left(\left|\sum\limits_{|\alpha|=n}\p_y^{\alpha}(f_1-f_2)\right|\frac{1}{\rho}
\right.\left.\left.+\left|\sum\limits_{|\alpha|=n-1}\p_y^{\alpha}
(f_1-f_2)\right|\frac{n}{\rho^2}\right)\rho^2{\rm d}\rho\right|^2 {\rm
d}\gamma{\rm d}\kappa\cr
\le& \int_s^{+\infty} \int_{\p B_R} \kappa^{4}\left|\int_0^{2\pi}{\rm
d}\theta\int_0^{\pi}\sin\varphi{\rm d}\varphi\int_{R-r}^{2R}\frac{1}{4\pi
\kappa^n}\right.\cr
&\left(\left|\sum\limits_{|\alpha|=n}\p_y^{\alpha}(f_1-f_2)\right|\frac{1}{R-r}
\right.\left.\left.+\left|\sum\limits_{|\alpha|=n-1}\p_y^{\alpha}
(f_1-f_2)\right|\frac{n}{(R-r)^2}\right)\rho^2{\rm d}\rho\right|^2 {\rm
d}\gamma{\rm d}\kappa\cr
=&\int_s^{+\infty} \int_{\p B_R} \kappa^{4}\left|\int_0^{2\pi}{\rm
d}\theta\int_0^{\pi}\sin\varphi{\rm d}\varphi\int_{0}^{+\infty}\frac{1}{4\pi
\kappa^n}\right.\cr
&\left(\left|\sum\limits_{|\alpha|=n}\p_y^{\alpha}(f_1-f_2)\right|\frac{1}{R-r}
\right.\left.\left.+\left|\sum\limits_{|\alpha|=n-1}\p_y^{\alpha}
(f_1-f_2)\right|\frac{n}{(R-r)^2}\right)\rho^2{\rm d}\rho\right|^2 {\rm
d}\gamma{\rm d}\kappa.
\end{align*}
Changing back to the Cartesian coordinates with respect to $y$, we have
\begin{align}
L_1
\le& \int_s^{+\infty} \int_{\p B_R}
\kappa^{4}\left|\int_{\mathbb{R}^3}\frac{1}{4\pi \kappa^n}\right.\cr
&\left(\left|\sum\limits_{|\alpha|=n}\p_y^{\alpha}(f_1-f_2)\right|\frac{1}{R-r}
\right.\left.\left.+\left|\sum\limits_{|\alpha|=n-1}\p_y^{\alpha}
(f_1-f_2)\right|\frac{n}{(R-r)^2}\right)dy\right|^2 {\rm d}\gamma{\rm
d}\kappa\cr
\le& C n\|f_1-f_2\|_{H^n(B_R)}\int_s^{+\infty}\kappa^{4-2n}{\rm d}\kappa\cr
=&C\frac{n}{2n-5}\|f_1-f_2\|_{H^n(B_R)}\frac{1}{s^{2n-5}}\cr
\label{3.10}\le &3C\|f_1-f_2\|_{H^n(B_R)}\frac{1}{s^{2n-5}},\quad n\ge 3.
\end{align}
Next we estimate $L_2$ for $d=3$,
\begin{align*}
L_2=&\int_s^{+\infty} \int_{\p B_R} \kappa^{2} |\p_{\nu}u(x,\kappa)|^2{\rm
d}\gamma{\rm d}\kappa\cr
=&\int_s^{+\infty} \int_{\p B_R}
\kappa^{4}\left|\int_{\mathbb{R}^3}\left(\na_y\frac{e^{{\rm
i}\kappa|x-y|}}{4\pi|x-y|}\cdot\nu\right)(f_1-f_2)(y){\rm d}y\right|^2 {\rm
d}\gamma{\rm d}\kappa.
\end{align*}
Noting that $\na_y\frac{e^{{\rm i}\kappa|x-y|}}{4\pi|x-y|}=-\na_x\frac{e^{{\rm
i}\kappa|x-y|}}{4\pi|x-y|}$ and ${\rm supp}f_j\subset B_R$, we have
\begin{align*}
L_2=&\int_s^{+\infty} \int_{\p B_R} \kappa^{2} |\p_{\nu}u(x,\kappa)|^2{\rm
d}\gamma{\rm d}\kappa\cr
=&\int_s^{+\infty} \int_{\p B_R}
\kappa^{2}\left|\int_{\mathbb{R}^3}\left(\na_y\frac{e^{{\rm
i}\kappa|x-y|}}{4\pi|x-y|}\cdot\nu\right)(f_1-f_2)(y){\rm d}y\right|^2 {\rm
d}\gamma{\rm d}\kappa\cr
=&\int_s^{+\infty} \int_{\p B_R}
\kappa^{2}\left|\int_{\mathbb{R}^3}\frac{e^{{\rm
i}\kappa|x-y|}}{4\pi|x-y|}\left(\na_y(f_1-f_2)(y)\cdot\nu\right){\rm
d}y\right|^2 {\rm d}\gamma{\rm d}\kappa.
\end{align*}
Following a similar argument as that for the proof of (\ref{3.10}), we can
obtain
\begin{align}\label{3.11}
L_2\le C
n\|f_1-f_2\|_{H^{n+1}(B_R)}\int_s^{+\infty}\kappa^{2-2n}{\rm
d}\kappa=C\frac{n}{2n-3} \|f_1-f_2\|_{ H^{n+1}(B_R)}\frac{1}{s^{2n-3}},\quad
n\ge 2.
\end{align}
Combining (\ref{3.10})--(\ref{3.11}) and noting $s>1$, we obtain (\ref{3.9}) for
$d=3$.

When $d=2$, we have
\begin{align*}
L_1=&\int_s^{+\infty} \int_{\p B_R} \kappa^{3}|u(x, \kappa)|^2 {\rm d}\gamma{\rm
d}\kappa\cr
=&\int_s^{+\infty} \int_{\p B_R} \kappa^{3}\left|\int_{\mathbb{R}^2}\frac{{\rm
i}}{4}H_0^{(1)}(\kappa|x-y|)(f_1-f_2)(y){\rm d}y\right|^2 {\rm d}\gamma{\rm
d}\kappa.
\end{align*}
The Hankel function can also be expressed by the following integral when $t>0$
(e.g., \cite{Wa}, Chapter VI):
\begin{align*}
H_0^1(t)=\frac{2}{{\rm i}\pi}\int_0^{+\infty}e^{{\rm
i}ts}(s^2-1)^{-1/2}{\rm d}s.
\end{align*}
Using the polar coordinates $\rho=|y-x|$ originated at $x$ with respect to $y$,
we have
\begin{align*}
L_1=\int_s^{+\infty} \int_{\p B_R} \kappa^{3}\left|\int_0^{2\pi}{\rm
d}\theta\int_0^{+\infty}\frac{1}{4}H_0^{(1)}(\kappa\rho)(f_1-f_2)\rho{\rm
d}\rho\right|^2 {\rm d}\gamma{\rm d}\kappa.
\end{align*}
Let
\begin{align}\label{3.12}
H_n(t)=\frac{2}{{\rm i}\pi}\int_0^{+\infty}\frac{e^{{\rm i}ts}}{({\rm
i}s)^n(s^2-1)^{1/2}}{\rm d}s, \quad n=1,2,\cdots.
\end{align}
It is clear to note that
\[
 H_0(t)=H_0^{(1)}(t)\quad\text{and}\quad \frac{{\rm d}
H_n(t)}{{\rm d}t}=H_{n-1}(t), ~t>0, ~n\in\mathbb{N}.
\]
Using integration by parts and noting ${\rm supp}f_j\subset B_r\subset B_R$, we obtain
\begin{align*}
L_1
=&\int_s^{+\infty} \int_{\p B_R} \kappa^{3}\left|\int_0^{2\pi}{\rm
d}\theta\int_{R-r}^{2R}\frac{H_1(\kappa\rho)}{4\kappa^2}\frac{\p
(f_1-f_2)\rho}{\p\rho}{\rm d}\rho\right|^2 {\rm d}\gamma{\rm d}\kappa\cr
=&\int_s^{+\infty} \int_{\p B_R} \kappa^{3}\left|\int_0^{2\pi}{\rm
d}\theta\int_{R-r}^{2R}\frac{H_n(\kappa\rho) }{4\kappa^{n+1}}\frac{\p^n
(f_1-f_2)\rho}{\p \rho^n}{\rm d}\rho\right|^2 {\rm d}\gamma{\rm d}\kappa.
\end{align*}
Consequently, we have
\begin{align*}
L_1\le&\int_s^{+\infty} \int_{\p B_R} \kappa^{3}\left|\int_0^{2\pi}{\rm
d}\theta\int_{R-r}^{2R}\left|\frac{H_n(\kappa\rho)
}{4\kappa^{n+1}}\right|\left|\frac{\p^n (f_1-f_2)\rho}{\p \rho^n}\right|{\rm
d}\rho\right|^2 {\rm d}\gamma{\rm d}\kappa\cr
\le&\int_s^{+\infty} \int_{\p B_R} \kappa^{3}\left|\int_0^{2\pi}{\rm
d}\theta\int_{R-r}^{2R}\left|\frac{H_n(\kappa\rho) }{4\kappa^{n+1}}\right|
\right.\cr
&\left(\left|\sum\limits_{|\alpha|=n}\p_y^{\alpha}
(f_1-f_2)\right|\right.\left.\left.+\left|\sum\limits_{|\alpha|=n-1}\p_y^{\alpha
}(f_1-f_2)\right|\frac{n}{\rho}\right)\rho{\rm d}\rho\right|^2 {\rm d}\gamma{\rm
d}\kappa\cr
\le&\int_s^{+\infty} \int_{\p B_R} \kappa^{3}\left|\int_0^{2\pi}{\rm
d}\theta\int_{R-r}^{2R}\left|\frac{H_n(\kappa\rho) }{4\kappa^{n+1}}\right|
\right.\cr
&\left(\left|\sum\limits_{|\alpha|=n}\p_y^{\alpha}
(f_1-f_2)\right|\right.\left.\left.+\left|\sum\limits_{|\alpha|=n-1}\p_y^{\alpha
}(f_1-f_2)\right|\frac{n}{R-r}\right)\rho{\rm d}\rho\right|^2 {\rm d}\gamma{\rm
d}\kappa.
\end{align*}
Noting (\ref{3.12}), we see that there exists a constant $C>0$ such that
$|H_n(\kappa\rho)|\le C$ for $n\ge 1$. Hence,
\begin{align*}
L_1\le&\int_s^{+\infty} \int_{\p B_R} \kappa^{3}\left|\int_0^{2\pi}{\rm
d}\theta\int_{R-r}^{2R}\frac{C }{4\kappa^{n+1}}
\right.\cr
&\left(\left|\sum\limits_{|\alpha|=n}\p_y^{\alpha}(f_1-f_2)\right|\right.\left.\left.
+\left|\sum\limits_{|\alpha|=n-1}\p_y^{\alpha}(f_1-f_2)\right|\frac{n}{R-r}
\right)\rho{\rm d}\rho\right|^2 {\rm d}\gamma{\rm d}\kappa.
\end{align*}
Changing back to the Cartesian coordinates with respect to $y$, we have
\begin{align}
L_1\le&\int_s^{+\infty} \int_{\p B_R}
\kappa^{3}\left|\int_{B_R}\frac{C}{4\kappa^{n+1}}
\right.\cr
&\left(\left|\sum\limits_{|\alpha|=n}\p_y^{\alpha}(f_1-f_2)\right|\right.\left.\left.
+\left|\sum\limits_{|\alpha|=n-1}\p_y^{\alpha}(f_1-f_2)\right|\frac{n}{R-r}
\right){\rm d}x\right|^2 {\rm d}\gamma{\rm d}\kappa\cr
\label{3.13}\le& C
n\|f_1-f_2\|_{H^n(B_R)}\int_s^{+\infty}\kappa^{1-2n}{\rm
d}\kappa=C\frac{n}{2n-2} \|f_1-f_2\|_ {H^n(B_R)}\frac{1}{s^{2n-2}}.
\end{align}
Next we estimate $L_2$ for $d=2$. A simple calculation yields
\begin{align*}
L_2=&\int_s^{+\infty} \int_{\p B_R} \kappa^{2} |\p_{\nu}u(x,\kappa)|^2{\rm
d}\gamma{\rm d}\kappa\cr
=&\int_s^{+\infty} \int_{\p B_R}
\kappa^{4}\left|\int_{\mathbb{R}^3}\left(\frac{\rm i}{4}\na_y
H_0^{(1)}(\kappa|x-y|)\cdot\nu\right)(f_1-f_2)(y){\rm d}y\right|^2 {\rm
d}\gamma{\rm d}\kappa.
\end{align*}
Noting that $\na_y H_0^{(1)}(\kappa|x-y|)=-\na_x H_0^{(1)}(k|x-y|)$ and
${\rm supp}f_j\subset B_r\subset B_R$, we have
\begin{align*}
L_2=&\int_s^{+\infty} \int_{\p B_R} \kappa^{2} |\p_{\nu}u(x,\kappa)|^2{\rm
d}\gamma{\rm d}\kappa\cr
=&\int_s^{+\infty} \int_{\p B_R}
\kappa^{2}\left|\int_{\mathbb{R}^3}\left(\frac{\rm
i}{4}\na_y H_0^{(1)}(\kappa|x-y|)\cdot\nu\right)(f_1-f_2)(y){\rm d}y\right|^2
{\rm d}\gamma{\rm d}\kappa\cr
=&\int_s^{+\infty} \int_{\p B_R}
\kappa^{2}\left|\int_{\mathbb{R}^3}\frac{\rm
i}{4}H_0^{(1)}(\kappa|x-y|)\left(\na_y(f_1-f_2)(y)\cdot\nu\right){\rm
d}y\right|^2 {\rm d}\gamma{\rm d}\kappa.
\end{align*}
Following a similar argument as the proof of (\ref{3.13}), we can obtain
\begin{align}
L_2
\le& C n\|f_1-f_2\|_{H^{n+1}(B_R)}\int_s^{+\infty}\kappa^{-2n}{\rm d}\kappa\cr
\label{3.14}=&C\frac{n}{2n-1}\|f_1-f_2\|_{H^{n+1}(B_R)}\frac{1}{s^{2n-1}}.
\end{align}
Combining (\ref{3.13}) and (\ref{3.14}) completes the proof of (\ref{3.9}) for
$d=2$.

\end{proof}

The following lemma is proved in \cite{CIL}.

\begin{lemm}
Let $J(z)$ be analytic in $S=\{z=x+{\rm i}y\in\mathbb{C}: -\frac{\pi}{4}<{\rm arg}
z<\frac{\pi}{4}\}$ and continuous in $\bar{S}$
satisfying
\[
 \begin{cases}
  |J(z)|\leq\epsilon, & z\in (0, ~ L],\\
  |J(z)|\leq V, & z\in S,\\
  |J(0)|=0.
 \end{cases}
\]
Then there exits a function $\mu(z)$ satisfying
\[
 \begin{cases}
  \mu(z)\geq\frac{1}{2},  & z\in(L, ~ 2^{\frac{1}{4}}L),\\
  \mu(z)\geq \frac{1}{\pi}((\frac{z}{L})^4-1)^{-\frac{1}{2}}, & z\in
(2^{\frac{1}{4}}L, ~ \infty)
 \end{cases}
\]
such that
\[
|J(z)|\leq V\epsilon^{\mu(z)}, \quad\forall\, z\in (L, ~ \infty).
\]
\end{lemm}

\begin{lemm}
 Let $f_j\in\mathcal{C}_M, j=1,2$. Then there exists a function
$\mu(z)$ satisfying
\begin{equation}\label{mu}
 \begin{cases}
  \mu(s)\geq\frac{1}{2}, \quad & s\in(K, ~ 2^{\frac{1}{4}}K),\\
  \mu(s)\geq \frac{1}{\pi}((\frac{s}{K})^4-1)^{-\frac{1}{2}}, \quad & s\in
(2^{\frac{1}{4}}K, ~\infty),
 \end{cases}
\end{equation}
such that
\[
 |I_1(s)+I_2(s)|\leq CM^2 e^{(4R+1)s}\epsilon^{2\mu(s)},\quad\forall s\in (K, ~\infty),
\]
for $d=2,3$.
\end{lemm}

\begin{proof}
 It follows from Lemma 3.2 that
\[
 |[I_1(s)+I_2(s)]e^{-(4R+1)s}|\leq CM^2,\quad\forall s\in S.
\]
Recalling \eqref{e1}, \eqref{3.1}-\eqref{3.4}, we have
\[
 |[I_1(s)+I_2(s)]e^{-(4R+1)s}|\leq\epsilon^2,\quad s\in [0, ~K].
\]
A direct application of Lemma 3.5 shows that there exists a function $\mu(s)$
satisfying \eqref{mu} such that
\[
 |[I_1(s)+I_2(s)]e^{-(4R+1)s}|\leq CM^2\epsilon^{2\mu},\quad\forall s\in (K, ~\infty),
\]
which completes the proof.
\end{proof}

Now we show the proof of Theorem 2.1.

\begin{proof}
 We can assume that $\epsilon<e^{-1}$, otherwise the estimate is
obvious. Let
\[
s=\begin{cases}
\frac{1}{((4R+3)\pi)^{\frac{1}{3}}}K^{\frac{2}{3}}|\ln\epsilon|^{\frac{1}{4}},
&
2^{\frac{1}{4}}
((4R+3)\pi)^{\frac{1}{3}}K^{\frac{1}{3}}<|\ln\epsilon|^{\frac{1}{4}},\\
K, &|\ln\epsilon|\leq 2^{\frac{1}{4}}((4R+3)\pi)^{\frac{1}{3}}K^{\frac{1}{3}}.
 \end{cases}
\]
If
$2^{\frac{1}{4}}(((4R+3)\pi)^{\frac{1}{3}}K^{\frac{1}{3}}<|\ln\epsilon|^{\frac{
1}{4}}$, then we have
\begin{align*}
 |I_1(s)+I_2(s)|&\leq CM^2 e^{(4R+3)s}
e^{-\frac{2|\ln\epsilon|}{\pi}((\frac{s}{K})^4-1)^{-\frac{1}{2}}}\cr
&\leq CM^2
e^{\frac{(4R+3)}{((4R+3)\pi)^{\frac{1}{3}}}K^{\frac{2}{3}}|\ln\epsilon|^{\frac{
1}{4}}-\frac{2|\ln\epsilon|}{\pi}
(\frac{K}{s})^2}\\
&=CM^2 e^{-2\left(\frac{(4R+3)^2}{\pi}\right)^{\frac{1}{3}}K^{\frac{2}{3}}
|\ln\epsilon|^{\frac{1}{2}}\left(1-\frac{1}{2}
|\ln\epsilon|^{-\frac{1}{4}}\right)}.
\end{align*}
Noting that $\frac{1}{2} |\ln\epsilon|^{-\frac{1}{4}}<\frac{1}{2}$,
$\left(\frac{(4R+3)^2}{\pi}\right)^{\frac{1}{3}}>1$  we have
\begin{align*}
 |I_1(s)+I_2(s)|&
\leq CM^2
e^{-K^{\frac{2}{3}}|\ln\epsilon|^{\frac{1}{2}}}.
\end{align*}
Using
the elementary inequality
\[
 e^{-x}\leq \frac{(6n-6d+3)!}{x^{3(2n-2d+1)}}, \quad x>0,
\]
we get
\begin{align}\label{3.16}
 |I_1(s)+I_2(s)|\leq\frac{CM^2}{\left(\frac{K^2|\ln\epsilon|^{\frac{3}{2}}}{
(6n-6d+3)^3}\right)^{2n-2d+1}}.
\end{align}
If $|\ln\epsilon|\leq
2^{\frac{1}{4}}(((4R+3)\pi)^{\frac{1}{3}}K^{\frac{1}{3}}$, then $s=K$. We have
from \eqref{e1}, \eqref{3.1}-\eqref{3.4}  that
\[
 |I_1(s)+I_2(s)|\leq \epsilon^2,
\]
Here we have noted that for $s>0$, $I_1(s)+I_2(s)=\int_0^s \int_{\p B_R}\kappa^{d-1} \bigl(|\p_{\nu}u(x,\kappa)|^2
+\kappa^2|u(x, \kappa)|^2 \bigr){\rm d}\gamma{\rm d}\kappa$.
Hence we obtain from Lemma 3.3 and \eqref{3.16} that
\begin{align*}
 &\int_0^\infty \int_{\p B_R}\kappa^{d-1} \bigl(|\p_{\nu}u(x,\kappa)|^2
+\kappa^2|u(x, \kappa)|^2 \bigr){\rm d}\gamma{\rm d}\kappa\\
&\leq I_1(s)+I_2(s)+\int_s^\infty \int_{\p B_R}\kappa^{d-1} \bigl(|\p_{\nu}u(x,\kappa)|^2
+\kappa^2|u(x, \kappa)|^2 \bigr){\rm d}\gamma{\rm d}\kappa\\
&\leq
\epsilon^2+\frac{CM^2}{\left(\frac{K^2|\ln\epsilon|^{\frac{3}{2}}}{(6n-6d+3)^3}
\right)^{2n-2d+1}}+\frac{C
\|f_1-f_2\|^2_{H^{n+1}(B_R)}}{\left(2^{-\frac{1}{4}}((4R+3)\pi)^{-\frac{1}{3}}K^{
\frac{2}{3}} |\ln\epsilon|^{\frac{1}{4}}\right)^{2n-2d+1}}.
\end{align*}
By Lemma 3.1, we have
\[
 \|f_1-f_2\|^2_{L^2(B_R)}\leq C \left(\epsilon^2
+\frac{M^2}{\left(\frac{K^2|\ln\epsilon|^{\frac{3}{2}}}{(6n-6d+3)^3}\right)^{
2n-2d+1}}+\frac{M^2}{\left(\frac{K^{\frac{2}
{3}}|\ln\epsilon|^{\frac{1}{4}}}{(R+1)(6n-6d+3)^3}\right)^{2n-2d+1}}\right).
\]
Since $K^{\frac{2}{3}}|\ln\epsilon|^{\frac{1}{4}}\leq K^2
|\ln\epsilon|^{\frac{3}{2}}$ when $K>1$ and $|\ln\epsilon|>1$, we obtain
the stability estimate.
\end{proof}

\end{document}